\documentclass[12pt]{amsart}
\usepackage{graphicx}
\usepackage{float}
\usepackage[dvipsnames]{xcolor}
\usepackage{hyperref}
\usepackage[shortlabels]{enumitem}
\usepackage{amsmath,amssymb,amsthm,amscd}
\newtheorem{theorem}{Theorem}
\newtheorem{lemma}[theorem]{Lemma}

\newtheorem*{corollary*}{Corollary}
\newtheorem{proposition}[theorem]{Proposition}
\newtheorem*{proposition*}{Proposition}

\theoremstyle{definition}
\newtheorem{remark}[theorem]{Remark}

\usepackage[margin=3cm]{geometry}
\newcommand{\Z}{\mathbb{Z}}
\newcommand{\F}{\mathbb{F}}
\newcommand{\Q}{\mathbb{Q}}
\newcommand{\Gal}{\textrm{Gal}}
\newcommand{\Aut}{\textrm{Aut}}
\newcommand{\ssm}{\smallsetminus}
\newcommand{\norleq}{\unlhd}
\newcommand{\im}{\textrm{im}}
\newcommand{\Cp}{\mathfrak{P}}
\newcommand{\GL}{\textrm{GL}}
\newcommand{\SL}{\textrm{SL}}

\newcommand{\oQ}{\overline{\Q}}
\newcommand{\tors}{\textrm{tors}}
\newcommand{\nr}{\text{nr}}
\newcommand{\subseteqc}{\overset{\sim}{\subseteq}}

\usepackage{mathrsfs}
\usepackage{url}
\date{}
\title[Uniform Bounds on Cyclotomic Division Fields]{Uniform Bounds on the Level of Cyclotomic Division Fields of Elliptic Curves}

\author{Sam Allen}
\address{The Ohio State University, Columbus, OH 43210, USA}
\email{allen.2767@buckeyemail.osu.edu}
\author{Tyler Genao}
\address{The Ohio State University, Columbus, OH 43210, USA}
\email{genao.5@osu.edu}
\begin{document}
\begin{abstract}
In this paper, we prove that for each number field $F$ there exists a uniform bound on the prime levels $p$ of elliptic curves $E/F$ for which $F(E[p])=F(\zeta_p)$. Under the Generalized Riemann Hypothesis, we also give uniform bounds on $p$ for which $F(E[p])/F$ is abelian, provided that $F$ has no rational complex multiplication. These are generalizations of results of Gonz\'{a}lez-Jim\'{e}nez and Lozano-Robledo \cite{GJLR16} to general number fields.
\end{abstract}
\maketitle

\section{Introduction}
Given a characteristic zero field $F$ and an elliptic curve $E/F$, for each integer $n>0$ the $n$-division field of $E$ over $F$, denoted $F(E[n])$, is the field obtained by adjoining to $F$ the $x$ and $y$-coordinates of all points $P\in E$ whose order divides $n$. The $n$-division field also appears as the fixed field under the action of the absolute Galois group $G_F:=\Gal(\overline{F}/F)$ on the $n$-torsion subgroup $E[n]$ of $E$. The $n$-division field is a finite Galois extension of $F$, with Galois group isomorphic to a subgroup of $\GL_2(\Z/n\Z)$, the general linear group of $2\times 2$ matrices over $\Z/n\Z$.
In analogy to the Kronecker-Weber theorem and roots of unity over $\Q$, the coordinates of $n$-torsion points of elliptic curves can lend an explicit description of class field theory for imaginary quadratic fields. 

For an elliptic curve $E/F$, the existence of the perfect, Galois-invariant $n$-Weil pairing on $E$ guarantees that $F(E[n])$ contains a primitive $n$'th root of unity $\zeta_n$ \cite[Corollary III.8.1.1]{Sil09}. A natural question is as follows: \textit{how much larger is the $n$-division field $F(E[n])$ compared to the $n$-cyclotomic field $F(\zeta_n)$?} Intuitively, the $n$-division field should be strictly larger ``most of the time" due to characteristics of the mod-$n$ Galois representation $\rho_{E,n}\colon G_F\rightarrow \GL_2(\Z/n\Z)$, which describes the action of $G_F$ on $E[n]$ and whose kernel has fixed field $F(E[n])$. For example, for any elliptic curve $E$ defined over its field of definition $F=\Q(j(E))$, if $E$ has CM (complex multiplication) then Hamakiotes and Lozano-Robledo showed that $F(E[n])=F(\zeta_n)$ implies $n\leq 3$ \cite{HLR24}. On the other hand, for any number field $F$ and any non-CM elliptic curve $E/F$, by Serre's Open Image Theorem there exists a constant $c:=c(E,F)>0$ such that for any prime $p>c$ one has $\rho_{E,p}(G_F)=\GL_2(\Z/p\Z)$ \cite{Ser72}, which forces $F(E[p])$ to be nonabelian.

In light of the above, one can ask if for each number field $F$ there exists a uniform bound on integers $n>0$ such that there exists an elliptic curve $E/F$ where $F(E[n])=F(\zeta_n)$. For example, as a consequence of progress towards Serre's uniformity question over $\Q$ (whether the constant $c:=c(E,\Q)$ in the preceding paragraph can be chosen independently of $E$) one has for all primes $p>37$ that $\Q(E[p])\neq \Q(\zeta_p)$.
In fact, 
Gonz\'{a}lez-Jim\'{e}nez and Lozano-Robledo have proven a \textit{sharp} upper bound over $\Q$ \cite{GJLR16}.
\begin{theorem}\cite[Theorem 1.1]{GJLR16}\label{Thm_SmallDivisionFieldsOverQ}
Let $E/\Q$ be an elliptic curve. If for $n\in\Z^+$ one has $\Q(E[n])=\Q(\zeta_n)$, then $n\leq 5$. More generally, if $\Q(E[n])/\Q$ is abelian then $n\in \lbrace 2,3,4,5,6,8\rbrace$, and $\emph{Gal}(\Q(E[n])/\Q)$ is isomorphic to 1 of 17 finite abelian groups.
\end{theorem}
Gonz\'{a}lez-Jim\'{e}nez and Lozano-Robledo proved their result through a close analysis of Galois representations of elliptic curves over $\Q$, in large part via group-theoretic considerations and explicit calculations with modular curves over $\Q$. However, given that considerably less is known about Galois representations over nontrivial number fields, many of their results do not immediately generalize beyond $\Q$.

In this paper, we initiate a study of cyclotomic division fields over general number fields $F$, where \textit{cyclotomic} means $F(E[n])=F(\zeta_n)$. Our first result shows that uniform bounds on prime levels of cyclotomic division fields exist over any number field $F$. We will use $\Delta_{F}$ to denote the discriminant of $F/\Q$.
\begin{theorem}\label{Thm_UniformBds_OnPrimeLevel_OfSmallDivisionFields}
Let $F$ be a number field. Then for all elliptic curves $E/F$ and for all primes $p\in\Z^+$ with $F(E[p])=F(\zeta_p)$, if $p\nmid \Delta_F$ and $p>13$ (resp. $p\mid \Delta_F$ and $p> 12[F:\Q]+1$),
then $E$ has an order $p$ point over an extension $K/F$ where $[K:F]\leq 6$ (resp. $[K:F]\leq 6[F:\Q]$). In particular, $p$ is uniformly bounded. 
\end{theorem}

We additionally prove more general results for when the mod-$p$ image is abelian. In the following, by GRH we mean the Generalized Riemann Hypothesis. Also recall that a number field is said to have rational CM if it contains an imaginary quadratic field.
\begin{theorem}\label{Thm_UniformBds_OnPrimeLevel_OfAbelianDivisionFields}
Let $F$ be a number field. Then for all elliptic curves $E/F$ and for all primes $p\geq 17$ where $F(E[p])/F$ is abelian, if $p\nmid \Delta_F$ or $p>(6 [F:\Q])!+1$, then $\rho_{E,p}(G_F)$ is contained in either the split or nonsplit Cartan subgroup mod-$p$. In particular, if GRH is true and $F$ has no rational CM, then $p$ is contained in a finite set $S_F$ from \cite[Theorem 1]{LV14} which depends only on $F$, and thus $p$ is uniformly bounded.
\end{theorem}
It is worth noting that if $F$ is a real number field, then one can mimic the proofs of \cite[Propositions 3.6, 3.7]{GJLR16} to show that if $p>2$ and $\rho_{E,p}(G_F)$ is an abelian subgroup of the normalizer of a Cartan subgroup, then it is also diagonalizable, which improves the first conclusion in Theorem \ref{Thm_UniformBds_OnPrimeLevel_OfAbelianDivisionFields}. However, their proof crucially uses the existence of a nontrivial complex conjugation in $G_F$ whose image in $\rho_{E,p}(G_F)$ has trace $0$ and determinant $-1$. For general number fields, one must work harder to prove this image is contained in a Cartan subgroup.

In contrast to Theorem \ref{Thm_SmallDivisionFieldsOverQ}, our proofs of Theorems \ref{Thm_UniformBds_OnPrimeLevel_OfSmallDivisionFields} and \ref{Thm_UniformBds_OnPrimeLevel_OfAbelianDivisionFields} utilize the group-theoretic condition that for an elliptic curve $E/F$, an $n$-division field is cyclotomic if and only if $\rho_{E,n}(G_F)\cap \SL_2(\Z/n\Z)=1$, where $\SL_2(\Z/n\Z)$ denotes the special linear group of $2\times 2$ matrices over $\Z/n\Z$. We also crucially apply Serre's results on the image of the inertia subgroup under the mod-$p$ Galois representation of an elliptic curve \cite{Ser72}, which is summarized in Theorem \ref{Thm_ImageOfInertia}. We combine this with a classification of subgroups of $\GL_2(\Z/p\Z)$ (Theorem \ref{Thm_DicksonClassification}) to put constraints ``above and below" the image $\rho_{E,p}(G_F)$, which by our cyclotomic or abelian hypothesis on $F(E[p])/F$ will prove that $p$ is uniformly bounded.

We end this introduction with some examples of elliptic curves with cyclotomic division fields. We link each elliptic curve's LMFDB page \cite{LMFDB}, as well as the page for their mod-$p$ Galois image if it exists. The subgroup numbering follows \cite{Sut16}.
\begin{itemize}[leftmargin=8pt]
    \item $(n=2)$, $F=\Q(\sqrt{2})$: the elliptic curve \href{https://www.lmfdb.org/EllipticCurve/2.2.8.1/98.1/a/5}{98.1-a5} given by
    \[
    E:{y}^2+{x}{y}+{y}={x}^{3}-36{x}-70
    \]
    satisfies $F(E[2])=F(\zeta_2)$. Its mod-$2$ image is  \href{https://beta.lmfdb.org/ModularCurve/Q/2.6.0.a.1/}{2Cs.}
    \item $(n=3)$, $F=\mathbb{Q}(\sqrt{-1}):$ the elliptic curve \href{https://www.lmfdb.org/EllipticCurve/2.0.4.1/722.1/a/3}{722.1-a3} given by 
    \[
    E:{y}^2+i{x}{y}+i{y}={x}^{3}+10{x}-90
    \]  
    satisfies $F(E[3])=F(\zeta_3)$. Its mod-$3$ image is \href{https://beta.lmfdb.org/ModularCurve/Q/3.24.0-3.a.1.1/}{3Cs.1.1.}
    \item $(n=5)$, $F=\mathbb{Q}(\sqrt{2})$: the elliptic curve \href{https://www.lmfdb.org/EllipticCurve/2.2.8.1/121.1/a/2}{121.1-a2} given by 
    \[
    E: {y}^2+{y}={x}^{3}-{x}^{2}-10{x}-20
    \] 
    satisfies $F(E[5])=F(\zeta_5)$. Its mod-$5$ image is \href{https://beta.lmfdb.org/ModularCurve/Q/5.120.0-5.a.1.2/}{5Cs.1.1.}
    \item $(n=5)$, $F=\Q(\zeta_{15})^+$: the elliptic curve \href{https://beta.lmfdb.org/EllipticCurve/4.4.1125.1/25.1/b/5}{25.1-b5} given by \[
    E: {y}^2+\left(a^{2}-1\right){y}={x}^3+\left(a^{3}-a^{2}-3a+2\right){x}^2+\left(-a+2\right){x}+a^{3}-2a
    \] 
    where $a$ is a root of $ x^{4} - x^{3} - 4 x^{2} + 4 x + 1 $, 
    satisfies $F(E[5])=F(\zeta_5)$. Its mod-$5$ image is 5Cs.1.1[2].
    \item $(n=5)$, $F=\mathbb{Q}(\sqrt{10})$: the elliptic curve \href{https://beta.lmfdb.org/EllipticCurve/2.2.40.1/121.1/d/2}{121.1-d2} given by \[
    E: {y}^2={x}^{3}+{x}^{2}-41{x}-199
    \] 
    satisfies $F(E[5])=F(\zeta_5)$.Its mod-$5$ image is \href{https://beta.lmfdb.org/ModularCurve/Q/5.120.0-5.a.1.1/}{5Cs.1.3.}
    \item $(n=7)$, $F= \mathbb{Q}(\sqrt{-3})$: the elliptic curve \href{https://beta.lmfdb.org/EllipticCurve/2.0.3.1/49.1/CMa/1}{49.1-CMa1} given by \[
    E: {y}^2+a{y}={x}^{3}+\left(a+1\right){x}^{2}+a{x}
    \] 
    where $a=\frac{1+\sqrt{-3}}{2}$,
    satisfies $F(E[7])=F(\zeta_7)$. Its mod-$7$ image is \href{https://beta.lmfdb.org/ModularCurve/Q/7.336.3-7.b.1.2/}{7Cs.1.1.}  
    \item $(n=7)$, $F=\mathbb{Q} (\zeta_{21})^+$: the elliptic curve \href{https://beta.lmfdb.org/EllipticCurve/6.6.453789.1/1.1/a/5}{1.1-a5} given by  \[
    E:{y}^2+(a^{5}-4a^{3}+a^{2}+2a-2){y}={x}^{3}+\left(a^{4}-3a^{2}+a+1\right){x}^{2}
    \]
    \[
    +(a^{5}+a^{4}-4a^{3}
    -2a^{2}+3a){x}+a^{5}-6a^{3}+2a^{2}+9a-6
    \]
    where $a$ is a root of $x^{6} - x^{5} - 6 x^{4} + 6 x^{3} + 8 x^{2} - 8 x + 1$, satisfies $F(E[7])=F(\zeta_7)$. Its mod-$7$ image is 7Cs.1.1[3].
\end{itemize}

\section*{Acknowledgments}
This project began through the Cycle undergraduate mathematics research program at The Ohio State University. The authors thank David Kruzel for earlier discussions involving this project. The preliminary calculations for mod-$p$ Galois images were done with \texttt{Magma} \cite{BCP93}.
\section{Preliminary Results}
In this section, we will review some key definitions, notation and results we will need for Galois representations of elliptic curves. 
\subsection{Definitions and notation}
Fix an algebraic closure $\overline{\Q}$ of $\Q$. Throughout this paper, we let $F$ denote an algebraic extension of $\Q$; thus $F\subseteq \oQ$. Let $E/F$ an elliptic curve, i.e., a nonsingular projective curve of genus one with a point over $F$. We can assume that $E$ is given by an equation
\[
E:y^2+a_1xy+a_3y=x^3+a_2x^2+a_4x+a_6
\]
where each $a_i\in F$. The set $E:=E(\oQ)$ of algebraic points on $E$ forms a group under a geometric chord and tangent law. There is a natural action of the absolute Galois group $G_F:=\Gal(\oQ/F)$ on $E$: for $\sigma\in G_F$ and $P:=(x,y)\in E$, we let $\sigma$ act on $P$ by $\sigma(P):=(\sigma(x),\sigma(y))$. Let $E[\tors]$ denote the subgroup of \textit{torsion points} of $E$ (points with finite order).  Then this action restricts to an action on $E[\tors]$; in fact, for each integer $n\in\Z^+$ we have an induced action on $E[n]$, the group of $n$-torsion points (points with order dividing $n$). This latter action is called the \textit{mod-$n$ Galois representation of $E$,} and is written as
\[
\rho_{E,n}\colon G_F\rightarrow \Aut(E[n]).
\]
The kernel of this action is $\Gal(\oQ/F(E[n]))$, where $F(E[n])$ is the \textit{$n$-division field of $E$;} this field is obtained by adjoining all coordinates of $n$-torsion points on $E$ to $F$. Thus $F(E[n])/F$ is a Galois extension, and we have a faithful action
\[
\rho_{E,n}\colon \Gal(F(E[n])/F)\hookrightarrow \Aut(E[n]).
\]

Recall that $E[n]$ is a free rank two $\Z/n\Z$-module \cite[Corollary III.6.4]{Sil09}. In particular, fixing a basis $\lbrace P,Q\rbrace$ of $E[n]$ given an isomorphism $\Aut(E[n])\cong \GL_2(n)$, where $\GL_2(n):=\GL_2(\Z/n\Z)$ is the group of $2\times 2$ invertible matrices over $\Z/n\Z$. Thus, we can write the mod-$n$ Galois representation as
\[
\rho_{E,n,P,Q}\colon G_F\rightarrow \GL_2(n).
\]
Explicitly, if for $\sigma\in G_F$ we have $\sigma(P)=aP+cQ$ and $\sigma(Q)=bP+dQ$, then $\rho_{E,n,P,Q}(\sigma)=\begin{bmatrix}
a&b\\
c&d
\end{bmatrix}$. Changing the basis of $E[n]$ from $\lbrace P,Q\rbrace$ conjugates this image. By abuse of notation, we will often write $\rho_{E,n}$ instead of $\rho_{E,n,P,Q}$, with some implicit basis $\lbrace P,Q\rbrace$ in mind.

In parallel to the above, recall that for an integer $n\in\Z^+$ there is an action of $G_F$ on the subgroup of $n$'th roots of unity $\mu_n\subseteq \oQ$. This action is described by the \textit{mod-$n$ cyclotomic character}
\[
\chi_n\colon G_F\rightarrow \Aut(\mu_n).
\]
Fixing a primitive $n$'th root $\zeta_n\in \mu_n$, this action is explicitly
\begin{align*}
\chi_n(\sigma):=a_\sigma&&\text{where }\sigma(\zeta_n)=\zeta_n^{a_\sigma}.
\end{align*}
The value $\chi_n(\sigma)$ is independent of the choice of $\zeta_n$. 

The Galois action on the $n$-torsion subgroup of an elliptic curve and the subgroup of $n$'th roots of unity are closely connected. For an elliptic curve $E/F$, a consequence of properties of the $n$-Weil pairing (see \cite[Proposition III.8.1]{Sil09}) is that the determinant of $\rho_{E,n}$ is the mod-$n$ cyclotomic character: for any $\sigma\in G_F$, one has
\[
\det\rho_{E,n}(\sigma)=\chi_n(\sigma).
\]
The Weil pairing also implies that $\zeta_n\in F(E[n])$, where $\zeta_n$ is a primitive $n$'th root of unity \cite[Corollary III.8.1.1]{Sil09}. Let us say that the division field $F(E[n])$ is \textit{cyclotomic} if $F(E[n])=F(\zeta_n)$.
\subsection{Subgroups of $\GL_2(p)$}
In his landmark paper, Serre \cite{Ser72} proved that for any number field $F$ and any non-CM elliptic curve $E/F$, one has for sufficiently large primes $p\gg_{E,F}0$ that $\rho_{E,p}(G_F)=\GL_2(p)$. An important step towards proving this was understanding what the action of inertia on $E[p]$ could be over a field of characteristic zero. In this section, we will review these results of Serre, since they play a crucial role in our analysis of cyclotomic division fields -- intuitively, these inertia images will be too large for our mod-$p$ Galois images if $p$ is unbounded.

Let us review a classification result for subgroups of $\GL_2(p)$ first, given by Serre \cite{Ser72} and essentially due to work of Dickson \cite{Dic58}. First, for an integer $n\in\Z^+$ we let $\SL_2(n):=\SL_2(\Z/n\Z)$ denote the subgroup of $\GL_2(n)$ of matrices with determinant $1$. Assume that $p\geq 3$ is prime; we define the \textit{split Cartan group mod-$p$} as the subgroup of diagonal matrices,
\[
C_s(p):=\left\lbrace \begin{bmatrix}
a&0\\
0&d\end{bmatrix}\in \GL_2(p)\right\rbrace.
\]
Its normalizer is denoted $N_s(p)$; explicitly, we have
\[
N_s(p)=C_s(p)\cup \begin{bmatrix}
0&1\\
1&0
\end{bmatrix}C_s(p)=\left\lbrace \begin{bmatrix}
a&0\\
0&d
\end{bmatrix}, \begin{bmatrix}
0&b\\
c&0
\end{bmatrix}\in \GL_2(p)\right\rbrace.
\]
We call elements of $C_s(p)$ \textit{diagonal,} and elements of $N_s(p)\ssm C_s(p)$ \textit{antidiagonal.}
Each diagonal matrix $m:=\begin{bmatrix}
a&0\\
0&d
\end{bmatrix}\in C_s(p)$ has a \textit{flip,} which is $m_f:=\begin{bmatrix}
d&0\\
0&a
\end{bmatrix}$. 
For a subgroup $H\subseteq C_s(p)$, we let $H_f$ denote the subgroup of all flips of elements from $H$.

We will also define the \textit{nonsplit Cartan group mod-$p$,} denoted $C_{ns}(p)$, as follows. We will write $\F_p:=\Z/p\Z$.
Fix the least positive integer $\epsilon$ which generates $\F_p^\times$: then we have 
\[
C_{ns}(p)=\left\lbrace \begin{bmatrix}
a&b\epsilon\\
b&a
\end{bmatrix}\in \GL_2(p)\right\rbrace.
\]
This is the regular representation of $\F_p[\sqrt{\epsilon}]^\times$ acting on itself via multiplication with respect to the basis $\lbrace 1,\sqrt{\epsilon}\rbrace$. Its normalizer is
\[
N_{ns}(p)=C_{ns}(p)\cup \begin{bmatrix}
1&0\\
0&-1
\end{bmatrix}C_{ns}(p)=\left\lbrace \begin{bmatrix}
a&b\epsilon\\
b&a
\end{bmatrix}, \begin{bmatrix}
a&b\epsilon\\
-b&-a
\end{bmatrix}\right\rbrace.
\]

Let us also define the \textit{Borel subgroup mod-$p$} as the subgroup of upper triangular matrices,
\[
B_0(p):=\left\lbrace \begin{bmatrix}
a&b\\
0&d
\end{bmatrix}\in \GL_2(p)\right\rbrace.
\]
We recall that for any subgroup $H\subseteq B_0(p)$, its \textit{semisimplification} is
\[
H^{\text{ss}}:=\left\lbrace \begin{bmatrix}
a&0\\
0&d
\end{bmatrix}\in \GL_2(p): \exists b\in \F_p\text{ such that }\begin{bmatrix}
a&b\\
0&d
\end{bmatrix}\in \GL_2(p)\right\rbrace.
\]
We define the \textit{semi-Cartan subgroup mod-$p$} as
\[
D:=D(p):=\left\lbrace \begin{bmatrix}
a&0\\
0&1
\end{bmatrix}\in \GL_2(p)\right\rbrace.
\]
Finally, we use $Z:=Z(p)$ to denote the subgroup of scalar matrices in $\GL_2(p)$. 

Here is the Serre-Dickson classification result. For subgroups $G,H\subseteq \GL_2(p)$, when we write $H\overset{\sim}{\subseteq} G$ we mean that $H$ is contained in $G$ up to conjugacy. 
\begin{theorem}[Classification of subgroups of $\GL_2(p)$]\cite{Dic58, Ser72}\label{Thm_DicksonClassification}
Fix a prime $p\in\Z^+$, and let $G\subseteq \emph{GL}_2(p)$ be a subgroup. Then one of the following holds:
\begin{enumerate}[a.]
\item $G$ contains $\emph{SL}_2(p)$ ($\emph{SL}_2$ case).
\item We have $G\overset{\sim}{\subseteq} B_0(p)$ (Borel case).
\item We have $G\overset{\sim}{\subseteq} N_s(p)$ or $G\overset{\sim}{\subseteq}N_{ns}(p)$ (normalizer of Cartan case).
\item The projective quotient $G/Z\cap G$ is isomorphic to $A_4$, $S_4$ or $A_5$ (exceptional cases).
\end{enumerate}
\end{theorem}
\subsection{Bounds on the level of Cartan Galois images}
To wrap this section up, we will describe a result of Larson and Vaintrob on isogeny characters of elliptic curves; this will give us uniform bounds on prime levels $p$ of abelian $p$-division fields, which is Theorem \ref{Thm_UniformBds_OnPrimeLevel_OfAbelianDivisionFields}. Recall that for an elliptic curve $E/F$, a finite cyclic subgroup $C\subseteq E(\oQ)$ is \textit{$F$-rational} if it is stable under the action of $G_F$, i.e., for each $R\in C$ and $\sigma\in G_F$ one has $\sigma(R)\in C$. Setting $n:=|C|$, the associated representation
\[
r\colon G_F\rightarrow \Aut(C)\cong (\Z/n\Z)^\times
\]
is called the \textit{isogeny character of $C$.} Explicitly, writing $C=\langle P\rangle$ one has for each $\sigma\in G_F$ that
\begin{align*}
r(\sigma):=a_\sigma&&\text{where }\sigma(P)=a_\sigma P.
\end{align*}
\begin{theorem}\cite[Theorem 1]{LV14}\label{Thm_LV}
For a number field $F$, there exists an effectively computable set $S_F$ of prime numbers such that the following holds. Let $p$ be a prime not in $S_F$, and let $E/F$ be an elliptic curve such that $E[p]\otimes_{\F_p} \overline{\F_p}$ is reducible, say with degree one associated character $r\colon G_F\rightarrow \overline{\F_p}^\times$. Then one of the following holds.
\begin{enumerate}[1.]
\item There exists a CM elliptic curve $E'/F$ with CM field $K\subseteq F$, and associated characters $\psi,\psi'\colon G_F\rightarrow\overline{\F_p}^\times$ such that up to conjugation
\[
\rho_{E',p}(G_F)\otimes_{\F_p}\overline{\F_p}=\emph{im}\begin{bmatrix}
\psi&0\\
0&\psi'
\end{bmatrix},
\]
where $\psi^{12}=r^{12}$.
\item GRH fails for $F[\sqrt{-p}]$, and one has
\[
r^{12}=\chi_p^6.
\]
Moreover, $E$ has an $F$-rational $p$-isogeny and $p\equiv 3\pmod 4$.
\end{enumerate}
\end{theorem}
Let us explain how Theorem \ref{Thm_LV} gives uniform bounds on the level of absolutely reducible mod-$p$ Galois images. Fix an elliptic curve $E/F$ and a prime $p\in\Z^+$, and set $G:=\rho_{E,p}(G_F)$. If $G\subseteqc B_0(p)$, then $E$ has an $F$-rational $p$-isogeny, i.e., there exists a $G_F$-stable subspace of $E[p]$. This implies that $E[p]$ is reducible, and so if $p\not\in S_F$, then Theorem \ref{Thm_LV} implies that $F$ contains a CM field or that GRH fails for $F[\sqrt{-p}]$. Similarly, if $G\subseteq C_{ns}(p)$, then $E[p]$ is irreducible over $\F_p$ but reducible over $\F_{p^2}$ (since $G\otimes \F_{p^2}$ is diagonalizable). Therefore, if $G$ is contained in $B_0(p)$, $C_s(p)$ or $C_{ns}(p)$ up to conjugacy, and both GRH is true and $F$ has no rational CM, then $p\in S_F$, and thus $p$ is uniformly bounded.
\section{Galois Representations and the Image of Inertia}
In this section, we will prove various results on Galois representations of elliptic curves. In particular, we will provide alternate group-theoretic characterizations of cyclotomic and abelian division fields of elliptic curves, and give a description of the mod-$p$ images of inertia subgroups for elliptic curves over local fields, following \cite{Ser72}. We will also give several explicit descriptions of abelian subgroups of $\GL_2(p)$.
\subsection{An alternate description of cyclotomic and abelian division fields}
Let us first characterize cyclotomic and abelian division fields in terms of images of Galois representations. The main result of this paper is Theorem \ref{Thm_UniformBds_OnPrimeLevel_OfSmallDivisionFields}, which says that for a number field $F$, for sufficiently large primes $p\gg_F 0$ and for all elliptic curves $E/F$, one has $F(E[p])\neq F(\zeta_p)$.
The guiding philosophy is that the mod-$p$ Galois image is ``relatively uniformly large" when $p$ is uniformly large. 
This has been demonstrated over $\Q$ by progress towards Serre's uniformity question \cite{LFL21, FL}, and over numbers fields without rational CM \cite{Gen24, IK24, Gen25}. 

Here is our characterization of cyclotomic and abelian division fields in terms of Galois representations. Note that a cyclotomic division field is automatically abelian.
\begin{proposition}\label{Prop_ConditionForAbelianAndCyclotomic}
Let $F/\Q$ be an algebraic extension, $E/F$ an elliptic curve and $n\in\Z^+$. Set $G:=\rho_{E,n}(G_F)$.
\begin{enumerate}[a.]
\item One has that $G$ is abelian if and only if $F(E[n])/F$ is abelian.
\item One has that $G\cap \emph{SL}_2(n)=1$ iff $F(E[n])=F(\zeta_n)$. In such a case, one has that $G$ is abelian, with $\det\colon G\rightarrow (\Z/n\Z)^\times$ being an injection.
\end{enumerate}
\end{proposition}
\begin{proof}
First, let us observe that the mod-$n$ Galois representation
\[
\rho_{E,n}\colon G_F\rightarrow \GL_2(n)
\]
has kernel $\Gal(\oQ/F(E[n]))$. Modding out by this kernel induces a faithful representation
\[
\rho_{E,n}\colon \Gal(F(E[n])/F)\hookrightarrow \GL_2(n),
\]
and thus we have $G\cong \Gal(F(E[n])/F)$. In particular, we have that $G$ is abelian if and only if $F(E[n])/F$ is abelian. This proves part a.

For part b., observe that the kernel of the determinant map $\det\colon G\rightarrow (\Z/n\Z)^\times$ is $G\cap \SL_2(n)$, and thus $G\cap \SL_2(n)$ can be identified with a subgroup of $\Gal(F(E[n])/F)$ whose fixed field $K$ is Galois over $F$. We claim that $K=F(\zeta_n)$. To see this, by the $n$-Weil pairing one has for all $\sigma\in G_F$ that
\[
\sigma(\zeta_n)=\zeta_n^{\det\rho_{E,n}(\sigma)}.
\]
It follows that an automorphism $\sigma\in G_F$ fixes $\zeta_n$ if and only if $\rho_{E,n}(\sigma)\in G\cap \SL_2(n)$, iff $\sigma$ fixes $K$. We deduce that $K=F(\zeta_n)$, and so
\[
G\cap \SL_2(n)\cong \Gal(F(E[n])/F(\zeta_n)).
\]
Therefore, we conclude that $G\cap \SL_2(n)=1$ iff $F(E[n])=F(\zeta_n)$. In this case, the kernel of $\det\colon G\rightarrow (\Z/n\Z)^\times$ is trivial. 
\end{proof}
\begin{remark}
Proposition \ref{Prop_ConditionForAbelianAndCyclotomic} shows that if an elliptic curve $E/F$ has a cyclotomic $n$-division field, then $\rho_{E,n}(G_F)$ is isomorphic to its image $\det\rho_{E,n}(G_F)$. If each prime $p\mid n$ is unramified in $F$, then one can show that the mod-$n$ cyclotomic character $\chi_n\colon G_F\rightarrow (\Z/n\Z)^\times$ is surjective, so from $\det\rho_{E,n}=\chi_n$ we conclude that $\rho_{E,n}(G_F)$ is isomorphic to $(\Z/n\Z)^\times$ under the determinant map.
\end{remark}
\subsection{The image of inertia \`{a} la Serre}
Let us review Serre's work on describing the image of the inertia subgroup under a mod-$p$ Galois representation of an elliptic curve \cite{Ser72}.
Fix a prime $p\in\Z^+$, and let $K$ be an algebraic extension of $\Q_p$; let $k$ denote the residue field of $K$. Also let $K_\nr$ denote the maximal unramified extension of $K$. Then the \textit{inertia group of $K$} is $I:=I_K:=\Gal(\overline{K}/K_{\nr})$. The inertia group fits into a short exact sequence
\[
1\rightarrow I\rightarrow G_K\rightarrow G_k\rightarrow 1.
\]
Let $K_t$ denote the maximal tamely ramified extension of $K$; thus $K\subseteq K_{\nr}\subseteq K_t$. Let $I_{p}:=I_{K,p}:=\Gal(\overline{K}/K_t)$ denote the \textit{wild inertia group of $K$,} and $I_{t}:=I_{K,t}:=I/I_p\cong \Gal(K^t/K^\nr)$ the \textit{tame inertia group of $K$.} We note that $I_p$ is the largest pro-$p$-group in $I$.

More can be said about $I_t$. Fixing a uniformizer $\pi$ of $K_\nr$, we know that $K_t$ is a compositum of the Galois extensions $K(\sqrt[d]{\pi})/K$ where each $d\in \Z^+$ is coprime to $p$. Therefore, $I_t$ can be identified as an inverse limit 
\[
I_t=\varprojlim_{d:\gcd(p,d)=1}\Gal(K(\sqrt[d]{\pi})/K).
\]
Thus, $I_t$ is a pro-cyclic group.
For each $d$ considered above, we have a natural isomorphism $\theta_d\colon \Gal(K(\sqrt[d]{\pi})/K)\xrightarrow{\sim} \mu_d$, where $\mu_d$ is the subgroup of $d$'th roots of unity in $\overline{\Q_p}$ (which is isomorphic to the subgroup of $d$'th roots of unity in $\overline{\F_p}$). In particular, we have an induced surjection $\theta_d\colon I_t\twoheadrightarrow \mu_d$. When $d=p^r-1$ for some $r>0$, we can identify $\mu_d\cong \F_{p^r}^\times$; in this case, we call $\theta_{p^r-1}\colon I_t\rightarrow \F_{p^r}^\times$ a \textit{fundamental character of level $r$.} These characters describe the action of the tame inertia group on $p$-torsion subgroups of elliptic curves \cite{Ser72}.

The following result is essentially from \cite{Ser72}, which was proven in the semistable case where the absolute ramification index $e:=e(K/\Q_p)$ is trivial; however, Serre essentially gave the ingredients for the proof where $e>1$. Lozano-Robledo proved an analogous result for elliptic curves over $\Q$ which are allowed to have additive reduction \cite{LR13}. Here, we state Serre's result as a mild generalization of both cases, where both additive reduction and a nontrivial absolute ramification index are allowed.
\begin{theorem}\cite[Theorem 3.2]{LR13}\label{Thm_ImageOfInertia}
Fix a prime $p\in\Z^+$ and a finite extension $K/\Q_p$ of degree $d$. Let $E/K$ be an elliptic curve. Then there exists a finite extension $L/K$ such that $E/L$ is semistable, whose ramification index satisfies $e:=e(L/K)\in \lbrace 1,2,3,4,6\rbrace$.

Let $e_0:=e(K/\Q)$ denote the absolute ramification index of $K$.
Let $I$ denote the inertia group of $L$, and $I_p$ the wild inertia group.
\begin{enumerate}[a.]
\item Suppose that $E/L$ has good ordinary or bad multiplicative reduction. Then $E$ has an $L$-rational $p$-isogeny. Fix a basis of $E[p]$ such that $\rho_{E,p}(G_L)$ is upper triangular. Then we have the following:
\begin{enumerate}[i.]
\item If the wild inertia group $I_p$ acts trivially on $E[p]$, then the mod-$p$ image of $I$ is semi-Cartan. More precisely, we have 
\[
\rho_{E,p}(I)=\emph{im}\begin{bmatrix}
\theta_{p-1}^{ee_0}&0\\
0&1
\end{bmatrix}=D^{ee_0}.
\]
\item If $I_p$ acts nontrivially on $E[p]$, then we have
\[
\rho_{E,p}(I)=\left\langle \emph{im}\begin{bmatrix}
\theta_{p-1}^{ee_0}&0\\
0&1
\end{bmatrix},\gamma\right\rangle=\left\langle D^{ee_0},\gamma\right\rangle,
\]
where $\gamma:=\begin{bmatrix}
1&1\\
0&1
\end{bmatrix}$. 
\end{enumerate}
\item Suppose that $E/L$ has good supersingular reduction.
\begin{enumerate}[i.]
\item If $I_p$ acts trivially on $E[p]$, then up to conjugacy the mod-$p$ image of $I$ is the $ee_0$'th power of the non-split Cartan group mod-$p$, i.e., we have
\[
\rho_{E,p}(I)=C_{ns}(p)^{ee_0}
\]
for an appropriate choice of basis for $E[p]$.
\item If $I_p$ acts nontrivially on $E[p]$, then with respect to an appropriate basis, we have that $\rho_{E,p}(I)\subseteq B_0(p)$. Furthermore, we have $\gamma\in \rho_{E,p}(I)$, and can write the semisimplification of $\rho_{E,p}(I)$ as
\[
\rho_{E,p}(I)^{\emph{ss}}=\emph{im}\begin{bmatrix}
\theta_{p-1}^{ee_0-e'}&0\\
0&\theta_{p-1}^{e'}
\end{bmatrix},
\]
where $e':=v(a_{p})$ is the valuation of the $p$'th coefficient $a_p$ in the formal group expansion of the multiplication-by-$p$ map on $E$, which satisfies $0\leq e'\leq ee_0$. In particular, we have $D^{(6d)!}\subseteq \rho_{E,p}(I)$.
\end{enumerate}
\end{enumerate}
\end{theorem}
\begin{remark}
As we will see when considering cyclotomic $p$-division fields, a quick size argument rules out the case where wild inertia acts nontrivially. However, wild inertia still needs consideration when studying abelian division fields.
\end{remark}

\subsection{Abelian subgroups of $\GL_2(p)$}
The way we prove Theorems \ref{Thm_UniformBds_OnPrimeLevel_OfSmallDivisionFields} and \ref{Thm_UniformBds_OnPrimeLevel_OfAbelianDivisionFields} involves analyzing the possible ``shapes" of subgroups of $\GL_2(p)$, using the maximal subgroup classification in Theorem \ref{Thm_DicksonClassification}. To this end, we first prove which shapes are permissible for an abelian subgroup of $\GL_2(p)$.

\begin{proposition}\label{Prop_AbelianSubgroupOf_GL2p_Restraints}
Let $p\geq 3$ be prime, and let $G\subseteq \emph{GL}_2(p)$ be a subgroup. \begin{enumerate}[a.]
\item If $G$ is abelian, then $G$ is contained in $B_0(p)$, $N_s(p)$ or $N_{ns}(p)$ up to conjugacy.
\item If $G\subseteqc B_0(p)$, then $G$ is diagonalizable iff $p\nmid \# G$. When $G\subseteq B_0(p)$ is abelian and $p\mid \# G$, one has $G=\langle \gamma, G\cap Z\rangle$ where $\gamma:=\begin{bmatrix}
1&1\\
0&1
\end{bmatrix}$.
\item For an abelian subgroup $G$ that is contained in $N_s(p)$ (resp. $N_{ns}(p)$) but not $C_s(p)$ (resp. $C_{ns}(p)$), one has $G=\langle n, G\cap Z\rangle$ for any element $n\in G\ssm C_s(p)$ (resp $n\in G\ssm C_{ns}(p)$).
\end{enumerate}
\end{proposition}

Before we prove this, let us prove a small lemma which allows us to analyze the two Cartan cases almost identically.
\begin{lemma}\label{Lemma_Index2_C_in_N}
Suppose that $N$ is a finite group with an index two subgroup $C\norleq N$. Let $G\subseteq N$ be any subgroup, and suppose that $G\not\subseteq C$. Then for any $n\in G\ssm C$ one has $G=\langle n, G\cap C\rangle$, as well as $\langle n\rangle\cap C=\langle n^2\rangle$, $2\mid |n|$ and $[G:G\cap C]=2$.
\end{lemma}
\begin{proof}
To begin with, suppose that $n\in G\ssm C$. Observe that for any $g\in G\ssm C$, the cosets $nC$ and $gC$ are equal, and thus $n^{-1}g\in C$. Since $n^{-1}g\in G$, we deduce that $n^{-1}g\in G\cap C$, i.e., $x\in \langle n, G\cap C\rangle$. This proves that $G=\langle n, G\cap C\rangle$.

Let us show that $\langle n\rangle \cap C=\langle n^2\rangle$. Since the coset $(nC)^2=1C$, we have $\langle n^2\rangle \subseteq \langle n\rangle\cap C$. For the other direction, we let $x\in \langle n\rangle\cap C$. From $x\in \langle n\rangle$, we can write $x=n^k$ for some $0\leq k<|n|$. If $k$ is odd, then we can write
\[
xC=(n^{k-1}C)\cdot (nC)=nC,
\]
which contradicts that $n\not\in C$. We deduce that $k$ is even, whence we conclude $\langle n\rangle\cap C=\langle n^2\rangle$. Next, from the formula
\begin{equation*}
|G|=\frac{|n|\cdot |G\cap C|}{|\langle n\rangle \cap (G\cap C)|}=\frac{|n|\cdot |G\cap C|}{|n^2|}
\end{equation*}
and from $|n^2|=\frac{|n|}{\gcd(2,|n|)}$, we also conclude that
\[
|G|=\gcd(2,|n|)\cdot |G\cap C|.
\]
From $G\not\subseteq C$ we have $G\cap C\neq G$, which forces $\gcd(2,|n|)=2$. We conclude that both $2\mid |n|$ and $[G:G\cap C]=2$.
\end{proof}

\begin{proof}[Proof of Proposition \ref{Prop_AbelianSubgroupOf_GL2p_Restraints}]
By Theorem \ref{Thm_DicksonClassification} on the classification of subgroups of $\GL_2(p)$, if $G$ is abelian, up to conjugacy $G$ is a subgroup of $B_0(p)$, $N_s(p)$ or $N_{ns}(p)$. This proves part a. For part b., we note that $\GL_2(p)$ acts by left multiplication on the $\F_p$-vector space $V:=\F_p\langle e_1,e_2\rangle$, where $e_1:= \begin{bmatrix}
1\\
0
\end{bmatrix}$ and $e_2:=\begin{bmatrix}
0\\
1
\end{bmatrix}$. Thus, any subgroup $G\subseteq \GL_2(p)$ also acts on $V$. By Maschke's Theorem, if $p\nmid \# G$ then $V$ is a semisimple $\F_p[G]$-module; in particular, any $\F_p[G]$-submodule $W\subseteq V$ has a complementary $\F_p[G]$-submodule $X\subseteq V$, so that $V=W\oplus X$. Thus, if $p\nmid \# G$ then $V$ has a $G$-stable subspace if and only if $G$ is diagonalizable. Therefore, if we assume that $G\subseteq B_0(p)$, then $\langle e_1\rangle$ is a $G$-stable subspace of $V$, whence we conclude that $p\nmid \# G$ iff $G$ is diagonalizable.

Suppose then that $G\subseteq B_0(p)$ is abelian and $p\mid\# G$. In the following two paragraphs, we will show that $G=\langle \gamma,G\cap Z\rangle$. First, we show $\gamma\in G$: by the preceding paragraph, we have $G\not\subseteq C_s(p)$, and so there exists $m=\begin{bmatrix}
a&b\\
0&d
\end{bmatrix}\in G$ with $b\neq 0$. We claim that we may choose $m\in G$ such that $a=d$: if this were not the case, then all such matrices in $G$ have distinct entries on their diagonal, and are thus individually diagonalizable; since $G$ is abelian, this then implies that all elements of $G$ are simultaneously diagonalizable, see e.g. \cite{Con} -- this is impossible. Thus, we may assume that $a=d$. It follows that $m^{p-1}=\begin{bmatrix}
1&e\\
0&1
\end{bmatrix}$ where $e:=(p-1)a^{p-2}b\neq 0$. Then choosing $x\in \Z$ with $ex\equiv 1\pmod p$, we deduce that $m^{(p-1)x}=\gamma\in G$. 

Next, we claim that $G=\langle G\cap C_s(p), \gamma\rangle$. This follows from the fact that for any $m=\begin{bmatrix}
a&b\\
0&d
\end{bmatrix}\in G$ with $b\neq 0$, one has
\[
m=\begin{bmatrix}
a&0\\
0&d
\end{bmatrix}\cdot \gamma^{x}
\]
where $x\in\Z$ is chosen so that $x\equiv a^{-1}b\pmod p$. This shows that $G\subseteq \langle G\cap C_s(p), \gamma\rangle$, whence the claim follows.
Finally, one checks for any $m=\begin{bmatrix}
a&0\\
0&d
\end{bmatrix}\in G$ that commutativity $\gamma \cdot m=m\cdot \gamma$ forces $a=d$, and thus $G\cap C_s(p)=G\cap Z$. We conclude that $G=\langle \gamma, G\cap Z\rangle$, which proves part b.

For part c., we first assume that $G\subseteq N_s(p)$ with $G\not\subseteq C_s(p)$. Then by Lemma \ref{Lemma_Index2_C_in_N}, we have for any element $n\in G\ssm C_s(p)$ that $G=\langle n, G\cap C_s(p)\rangle$. It suffices to show that $G\cap C_s(p)=G\cap Z$. This is similar to part b.: writing $n=\begin{bmatrix}
0&b\\
c&0
\end{bmatrix}$, we have for all $m=\begin{bmatrix}
a&0\\
0&d
\end{bmatrix}\in G\cap C_s(p)$ that $nm=mn$, which forces $a=d$. This proves that $G=\langle n, G\cap Z\rangle$. The case where $G\subseteq N_{ns}(p)$ with $G\not\subseteq C_{ns}(p)$ is near identical by Lemma \ref{Lemma_Index2_C_in_N}: one checks here that for any $n=\begin{bmatrix}
a&b\epsilon \\
-b&-a
\end{bmatrix}\in G$ and $m=\begin{bmatrix}
c&d\epsilon\\
d&c
\end{bmatrix}\in G\cap C_{ns}(p)$, one has from $nm=mn$ that $d=0$, and thus $m\in G\cap Z$, which shows that $G=\langle n, G\cap C_{ns}(p)\rangle =\langle n, G\cap Z\rangle$.
\end{proof}
In our proofs of Theorems \ref{Thm_UniformBds_OnPrimeLevel_OfSmallDivisionFields} and \ref{Thm_UniformBds_OnPrimeLevel_OfAbelianDivisionFields}, we will often need to argue that the image of inertia is an explicitly known subgroup of $\rho_{E,p}(G_F)$, not just up to conjugacy. The following lemma lets us do this via the classification of subgroups of $\GL_2(p)$ given in Theorem \ref{Thm_DicksonClassification}. This lemma focuses on the case where the image contains a conjugate of the power of the semi-Cartan group or nonsplit Cartan subgroup mod-$p$.
\begin{lemma}\label{Lemma_SemiCartanAndNonsplitCartan_inAbelianSubgps}
Fix a prime $p\geq 3$ and an integer $k>0$. Set $\gamma:=\begin{bmatrix}
1&1\\
0&1
\end{bmatrix}\in \emph{GL}_2(p)$.
\begin{enumerate}[a.]
\item If $D^k\subseteqc \langle \gamma, Z\rangle$ then $p-1\mid k$.
\item For any subgroup $G\subseteq C_s(p)$, if $D^k\subseteqc G$ and $p-1\nmid k$ then $D^k\subseteq G$ or $D^k_f\subseteq G$. If $D^k\subseteqc N_s(p)\ssm C_s(p)$, then $p-1\mid 2k$.
\item If $D^k\subseteqc C_{ns}(p)$ then $p-1\mid k$, and if $D^k\subseteqc N_{ns}(p)$ then $p-1\mid 2k$.
\item If $C_{ns}(p)^k\subseteqc \langle \gamma, Z\rangle$ then $p+1\mid \gcd(k, p^2-1)$.
\end{enumerate}
\end{lemma}
\begin{proof}
For this proof, let us fix a generator $g$ of $\F_p^\times$; then it follows that $n:=\begin{bmatrix}
g^k&0\\
0&1
\end{bmatrix}$ generates $D^k$, and has order $\frac{p-1}{\gcd(k,p-1)}$.

For part a., suppose that $D^k\subseteqc \langle \gamma, Z\rangle$. Then for some $m=\begin{bmatrix}
a&b\\
c&d
\end{bmatrix}\in \GL_2(p)$, we have $mD^km^{-1}\subseteq \left\lbrace \begin{bmatrix}
x&xy\\
0&x
\end{bmatrix}\in \GL_2(p)\right\rbrace$. We check that
\begin{equation}\label{Eqn_ConjugateSemiCartan}
mnm^{-1}=\begin{bmatrix}
adg^k-bc&-ab(g^k-1)\\
cd(g^k-1)&ad-bcg^k
\end{bmatrix}\frac{1}{ad-bc}.
\end{equation}
This forces $cd(g^k-1)=0$. Since $|g^k|=\frac{p-1}{\gcd(k,p-1)}$, if $p-1\nmid k$ then $\gcd(k,p-1)\neq p-1$, and thus $g^k\neq 1$, which implies that $cd=0$. If $c=0$, then we see that 
\[
mnm^{-1}=\begin{bmatrix}
g^k&-bd^{-1}(g^k-1)\\
0&1
\end{bmatrix},
\]
which means $g^k=1$, a contradiction. Thus $c\neq 0$, which implies that $d=0$, and so
\[
mnm^{-1}=\begin{bmatrix}
1&ac^{-1}(g^k-1)\\
0&g^k
\end{bmatrix},
\]
again an impossibility. We conclude that $p-1\mid k$.

For part b., writing $mnm^{-1}\subseteq G\subseteq N_s(p)$ for some $m:=\begin{bmatrix}
a&b\\
c&d
\end{bmatrix}\in \GL_2(p)$, we have that $mnm^{-1}$ is either a diagonal or antidiagonal matrix. Assuming that $p-1\nmid k$, we have $g^k\neq 1$; and so \eqref{Eqn_ConjugateSemiCartan} implies that either $ab=cd=0$ (diagonal case), or $adg^k=bc$ and $bcg^k=ad$ (antidiagonal case). 

In the diagonal case, if $c=0$ then invertibility of $m$ implies that $ad\neq 0$, and thus $b=0$, which shows by \eqref{Eqn_ConjugateSemiCartan} that
\[
mnm^{-1}=\begin{bmatrix}
g^k&0\\
0&1
\end{bmatrix}=m;
\]
thus $mD^{k}m^{-1}=D^k\subseteq G$. On the other hand, if $c\neq 0$ then $d=0$, and thus $bc\neq 0$ and $a=0$. This shows that
\[
mnm^{-1}=\begin{bmatrix}
1&0\\
0&g^k
\end{bmatrix}=m_f.
\]
This implies that $mD^km^{-1}=D^k_f\subseteq G$. 

In the antidiagonal case, we have both $adg^{2k}=ad$ and $bcg^{2k}=bc$. If $a=0$ or $d=0$, then $bc\neq 0$; we thus have $ad\neq 0$ or $bc\neq 0$, so that $g^{2k}=1$, and hence $p-1\mid 2k$. This concludes part b.

For part c., suppose that $D^k\subseteqc N_{ns}(p)$. The eigenvalues of elements of $D^k$ come in pairs $(x^k, 1)$ where $x\in \F_p^\times$; on the other hand, eigenvalues of elements from $C_{ns}(p)$ come in pairs $u\pm v\sqrt{\epsilon}$, while eigenvalues from elements of $N_{ns}(p)\ssm C_{ns}(p)$ come in pairs $\pm\sqrt{u^2-v^2\epsilon}$ (N.B.: in both cases we assume $(u,v)\neq (0,0)\in \F_p^2$). In particular, from $D^k\subseteqc N_{ns}(p)$ we know that for some $(u,v)\in \F_p^2\ssm \lbrace (0,0)\rbrace$ we have $g^k=u+v\sqrt{\epsilon}$ and $1=u-v\sqrt{\epsilon}$, or (without loss of generality) $g^k=\sqrt{u^2-v^2\epsilon}$ and $1=-\sqrt{u^2-v^2\epsilon}$. In the former case, this forces $u=1$ and $v=0$, and thus $g^k=1$, which implies that $p-1\mid k$. In the latter case, we deduce that $g^{2k}=1$, and thus $p-1\mid 2k$.

Finally, for part d. we observe that elements of $\langle \gamma, Z\rangle$ have repeated eigenvalues. As noted in part c., elements of $C_{ns}(p)$ have eigenvalues of the form $a\pm b\sqrt{\epsilon}$, which forces elements of any conjugate of $C_{ns}(p)^k$ in $\langle \gamma, Z\rangle$ to be contained in $Z$. This implies that $\frac{p^2-1}{\gcd(k, p^2-1)}\mid p-1$, and so $p+1\mid \gcd(k, p^2-1)$.
\end{proof}
\section{Proofs of Theorems \ref{Thm_UniformBds_OnPrimeLevel_OfSmallDivisionFields} and \ref{Thm_UniformBds_OnPrimeLevel_OfAbelianDivisionFields}}
In this section, we prove Theorems \ref{Thm_UniformBds_OnPrimeLevel_OfSmallDivisionFields} and \ref{Thm_UniformBds_OnPrimeLevel_OfAbelianDivisionFields}. In both proofs, we let $d:=[F:\Q]$. We start with proving the more general abelian case, which is Theorem \ref{Thm_UniformBds_OnPrimeLevel_OfAbelianDivisionFields}.
\subsection{A proof of Theorem \ref{Thm_UniformBds_OnPrimeLevel_OfAbelianDivisionFields}}
By Proposition \ref{Prop_AbelianSubgroupOf_GL2p_Restraints}, for an appropriate choice of basis $\lbrace P,Q\rbrace$ for $E[p]$ we have that $G$ is contained in $B_0(p), N_s(p)$ or $N_{ns}(p)$. 

\textbf{Case 1,} where $G\subseteq B_0(p)$. We claim this implies that $G\subseteq C_s(p)$. 
Suppose this were not true. Then by Proposition \ref{Prop_AbelianSubgroupOf_GL2p_Restraints}.b we have $G=\langle \gamma,G\cap Z\rangle$ where $\gamma:=\begin{bmatrix}
    1&1\\
    0&1
\end{bmatrix}$. We consider the two cases for ramification of $p$ in $F$.
\begin{enumerate}
\item Suppose that $p\nmid \Delta_F$. We can write
\[
G=\im \begin{bmatrix}
r&*\\
0&r
\end{bmatrix}
\]
where $r\colon G_F\rightarrow \F_p^\times$ is the isogeny character of $\langle P\rangle$. However, we then have $r^2(G_F)=\det(G)=\chi_p(G_F)$, and consequently $\chi_p(G_F)\subseteq (\F_p^\times)^2$. Since $p$ is unramified in $F$, we must also have $\chi_p(G_F)=\F_p^\times$, which forces $p=2$.
\item Suppose that $p\mid \Delta_F$. Write $e_0:=e(\Cp\mid p)$ where $\Cp$ is a prime of $F$ above $p$. By Theorem \ref{Thm_ImageOfInertia}, we know that up to conjugacy $G$ contains $D^{ee_0}$, $C_{ns}(p)^{ee_0}$ or $D^{6\cdot d!}$.

If $E$ has potential good ordinary or bad multiplicative reduction at a prime above $p$, then by Theorem \ref{Thm_ImageOfInertia} we have $D^{ee_0}\subseteqc G$. Thus, by Lemma \ref{Lemma_SemiCartanAndNonsplitCartan_inAbelianSubgps}.a we have $p-1\mid ee_0$; since $e\in \lbrace 1,2,3,4,6\rbrace$ and $e_0\leq d$, this implies that $p\leq 6d+1$. If $E$ has potential good supersingular reduction with trivial wild inertia, then Theorem \ref{Thm_ImageOfInertia} implies that $C_{ns}(p)^{ee_0}\subseteqc G$. Thus, Lemma \ref{Lemma_SemiCartanAndNonsplitCartan_inAbelianSubgps}.d implies that $p+1\mid \gcd(ee_0, p^2-1)$,
which means $p\leq 6d-1$. Finally, in the potential supersingular case where wild inertia acts nontrivially, we have by Theorem \ref{Thm_ImageOfInertia} that $D^{(6d)!}\subseteqc G$, and thus by Lemma \ref{Lemma_SemiCartanAndNonsplitCartan_inAbelianSubgps}.a we get $p-1\mid(6d)!$, which means $p\leq (6d)!+1$. 
\end{enumerate}
By the above analysis, we conclude that if $G\subseteq B_0(p)$, and either $p>2$ with $p\nmid \Delta_F$, or $p>(6d)!+1$, then $G\subseteq C_s(p)$.

\textbf{Case 2,} where $G\subseteq N_s(p)$. We will show that $G\subseteq C_s(p)$ when $p$ is large enough. Since $p\nmid \# G$, by Theorem \ref{Thm_ImageOfInertia} we have $D^{ee_0}\subseteqc G$ or $C_{ns}(p)^{ee_0}\subseteqc G$, with notation as in Case 1.

Suppose that $G\not\subseteq C_{s}(p)$. Then by Proposition \ref{Prop_AbelianSubgroupOf_GL2p_Restraints}.c, we have $G=\langle n,G\cap Z\rangle$ for any $n\in G\ssm C_{s}(p)$. 
\begin{enumerate}
\item Suppose that $p\nmid \Delta_F$. Then $e_0=1$, and so $D^e\subseteqc G$ or $C_{ns}(p)^e\subseteqc G$. if $D^e\subseteqc G$ then Lemma \ref{Lemma_SemiCartanAndNonsplitCartan_inAbelianSubgps}.b implies that $p-1\mid 2e$, and thus $p\leq 13$.
If $C_{ns}(p)^e\subseteqc G$, then Lemma \ref{Lemma_SemiCartanAndNonsplitCartan_inAbelianSubgps}.d shows that $p+1\mid \gcd(e,p^2-1)$, which implies that $p\leq 7$.
\item Suppose that $p\mid \Delta_F$. Similar to the previous part, in the case $D^{ee_0}\subseteqc G$ we get $p-1\mid 2ee_0$, and in the case $C_{ns}(p)^{ee_0}\subseteqc G$ we have $p+1\mid \gcd(ee_0, p^2-1)$.
\end{enumerate}
We conclude that when $G\subseteq N_s(p)$, if  $p\nmid \Delta_F$ and $p>13$, or if $p>12d+1$, then $G\subseteq C_s(p)$ with $D^{ee_0}\subseteqc G$. 

\textbf{Case 3,} where $G\subseteq N_{ns}(p)$. We claim that $G\subseteq C_{ns}(p)$. Again by Theorem \ref{Thm_ImageOfInertia}, we know that $D^{ee_0}\subseteqc G$ or $C_{ns}(p)^{ee_0}\subseteq G$. By Lemma \ref{Lemma_SemiCartanAndNonsplitCartan_inAbelianSubgps}.c and an analysis identical to that of Case 2, we can show that if $D^{ee_0}\subseteqc G$, then $p\leq 13$ when $p\nmid \Delta_F$, and otherwise $p\leq 12d+1$. Therefore, for primes $p$ where $p\nmid \Delta_F$ and $p>13$, or where $p>12d+1$, we must have $C_{ns}(p)^{ee_0}\subseteqc G$. In this case, if $G\not\subseteq C_{ns}(p)$ then Proposition \ref{Prop_AbelianSubgroupOf_GL2p_Restraints}.c implies that $G=\langle \gamma,G\cap Z\rangle$, which by Lemma \ref{Lemma_SemiCartanAndNonsplitCartan_inAbelianSubgps}.d forces $p\leq 5$ when $p\nmid \Delta_F$, and $p\leq 6d-1$ otherwise -- both of which are impossible. 

We conclude that when $G\subseteq N_{ns}(p)$, if $p\nmid \Delta_F$ and $p>13$, or if $p>12d+1$, then $G\subseteq C_{ns}(p)$. Furthermore, we have $C_{ns}(p)^{ee_0}\subseteq G$, not just up to conjugation, since $C_{ns}(p)$ is cyclic. This proves Theorem \ref{Thm_UniformBds_OnPrimeLevel_OfAbelianDivisionFields}.

\subsection{A proof of Theorem \ref{Thm_UniformBds_OnPrimeLevel_OfSmallDivisionFields}}
Assume that $F(E[p])=F(\zeta_p)$; then by Proposition \ref{Prop_ConditionForAbelianAndCyclotomic}, we have for $G:=\rho_{E,p}(G_F)$ that $G\cap \SL_2(p)=1$. Since $G$ is abelian, by Proposition \ref{Prop_AbelianSubgroupOf_GL2p_Restraints} it must be contained in $B_0(p), N_s(p)$ or $N_{ns}(p)$. Since $G$ embeds into $\F_p^\times$ via the determinant map, we have $p\nmid \# G$, which simplifies our analysis. For example, Theorem \ref{Thm_ImageOfInertia} shows that we have $D^{ee_0}\subseteqc G$ or $C_{ns}(p)^{ee_0}\subseteqc G$ for some $e\in \lbrace 1,2,3,4,6\rbrace$ and for any ramification index $e_0$ for a prime in $F$ above $p$.

We will show that $D^{ee_0}\subseteq G\subseteq C_s(p)$. We first claim that $D^{ee_0}\subseteqc G$. Suppose this is not true; then $C_{ns}(p)^{ee_0}\subseteqc G$, and so comparing sizes yields 
\[
p^2-1\mid \gcd(ee_0, p^2-1)\cdot \# G\mid \gcd(ee_0, p^2-1)\cdot (p-1),
\]
whence we have
\[
p+1\mid \gcd(ee_0, p^2-1).
\]
If $p\nmid \Delta_F$ then $e_0=1$, and thus $p\leq 5$. Otherwise, from $e_0\leq d$ we have $p\leq 6d-1$. We conclude that if $p\nmid \Delta_F$ and $p>5$, or if $p\mid \Delta_F$ and $p>6d-1$, then $D^{ee_0}\subseteqc G$. 

By Theorem \ref{Thm_DicksonClassification}, we know that an appropriate choice of basis for $E[p]$ implies $G$ is contained in $B_0(p)$, $N_s(p)$ or $N_{ns}(p)$. We break this up into three cases.

\textbf{Case 1:} $G\subseteq B_0(p)$. By Proposition \ref{Prop_AbelianSubgroupOf_GL2p_Restraints}.b, since $p\nmid \# G$ we have $G\subseteqc C_s(p)$.

\textbf{Case 2:} $G\subseteq N_s(p)\ssm C_s(p)$. Then $D^{ee_0}\subseteqc G\subseteq N_s(p)$. If $e_0=1$, then by Lemma \ref{Lemma_SemiCartanAndNonsplitCartan_inAbelianSubgps}.b we have $p-1\mid 2e$, and thus $p\leq 13$; otherwise, we have $p-1\mid 2ee_0$, and thus $p\leq 12d+1$. 

\textbf{Case 3:} $G\subseteq N_{ns}(p)$. Then by Lemma \ref{Lemma_SemiCartanAndNonsplitCartan_inAbelianSubgps}.c we have $p-1\mid 2ee_0$, which gives the same conclusions as in Case 2.

We conclude that when $G\cap \SL_2(p)=1$, if $p\nmid \Delta_F$ and $p>13$, or if $p\mid\Delta_F$ and $p>12d+1$, then for some basis $\lbrace P,Q\rbrace$ of $E[p]$ we have $D^{ee_0}\subseteqc G\subseteq C_s(p)$. By Lemma \ref{Lemma_SemiCartanAndNonsplitCartan_inAbelianSubgps}.b, we can assume that $D^{ee_0}\subseteq G$, without loss of generality. Since $G$ is cyclic, we know that $G^{ee_0}\subseteq D^{ee_0}$. Let us prove the following lemma, which shows there exists an extension $K/F$ of degree $[K:F]=[G:G^{ee_0}]\mid ee_0$ for which $\rho_{E,p,P,Q}(G_{K})\subseteq D$.
\begin{lemma}\label{Lemma_Subgroup_IsAttained_byBaseChange}
Let $E/F$ be an elliptic curve, and $n\in\Z^+$. Then for any subgroup $H\subseteq \rho_{E,n}(G_F)$, there exists a (unique) finite subextension $F(E[n])/K/F$ for which $\rho_{E,n}(G_K)=H$.
\end{lemma}
\begin{proof}
Since $F(E[n])/F$ is a finite Galois extension, there exists a unique subextension $F(E[n])/K/F$ for which $H=\Gal(F(E[n])/K)$. For the base change $E/K$ one has the Galois representation
\[
\rho_{E/K,n}\colon G_K\rightarrow \GL_2(n),
\]
where we use the same implicit basis of $E[n]$ for both $\rho_{E/F,n}$ and $\rho_{E/K,n}$.
We find that the kernel of this representation is
\[
\ker\rho_{E/K,n}=\Gal(\oQ/K(E[n])).
\]
Since $K\subseteq F(E[n])$, it follows that $K(E[n])=F(E[n])$, which shows that $\ker\rho_{E/K,n}=\ker\rho_{E/F,n}$. We thus have
\[
\rho_{E/K,n}(G_K)\cong G_K/\Gal(\oQ/F(E[n]))\cong \Gal(F(E[n])/K)=H.\qedhere
\]
\end{proof}
Continuing the proof of Theorem \ref{Thm_UniformBds_OnPrimeLevel_OfSmallDivisionFields}, we deduce that the basis element $Q$ of $E[p]$ is $K$-rational. If we let $b:=b([K:\Q])$ denote a strong uniform bound on prime orders of torsion points over $K$ (such as the one given by Parent \cite{Par99}), then we have $p\leq b$. We conclude that $E$ has an order $p$ point whose field of definition lies in an extension of $F$ of degree at most $6$ when $p\nmid \Delta_F$ and $p>13$, and of degree $\leq 6d$ when $p\mid \Delta_F$ and $p>12d+1$. This proves the claimed uniform bounds on primes levels of cyclotomic division fields.
\bibliographystyle{alpha}
\bibliography{bibfile}
\end{document}